\documentclass[12pt]{amsart}
\usepackage[lite]{amsrefs} 
\usepackage{amssymb,amsmath,enumitem}
\usepackage{graphicx}
\newcommand{\R}{\mathbb{R}}

\newcommand{\M}{\mathcal{M}}

\newcommand{\A}{\mathcal{A}}

\newcommand{\W}{{\bf{W}}}
\newcommand{\V}{{\bf{V}}}
\newcommand{\I}{{\bf{I}}}

\numberwithin{equation}{section}
\usepackage{hyperref}
\hypersetup{colorlinks=true, citecolor=blue, linkcolor=blue, urlcolor=blue}
\theoremstyle{plain}
\newtheorem{Thm}{Theorem}[section]
\newtheorem{Cor}[Thm]{Corollary}
\newtheorem{Lem}[Thm]{Lemma}

\theoremstyle{definition}

\newtheorem{Rem}[Thm]{Remark}

\theoremstyle{remark}

\begin{document}
\title[Wolff's inequality for intrinsic nonlinear potentials]
{Wolff's inequality for intrinsic nonlinear potentials and quasilinear elliptic equations}

\author{Igor E. Verbitsky}
\address{Department of Mathematics, University of Missouri,  Columbia,   \newline 
Missouri 65211, USA}
\email{\href{mailto:verbitskyi@missouri.edu}{verbitskyi@missouri.edu}}
\begin{abstract}
We prove an analogue of Wolff's inequality 
for the so-called intrinsic nonlinear potentials associated with 
the quasilinear elliptic equation
\[ 
-\Delta_{p} u  = \sigma u^{q}  \quad \text{in} \;\; \R^n, 
\] 
in the sub-natural growth case $0<q< p-1$, 
where $\Delta_{p}u = \text{div}( |\nabla u|^{p-2} \nabla u )$ is the $p$-Laplacian, 
and $\sigma$ is a nonnegative measurable function (or measure) on $\R^n$. 

As an application, we give a necessary and sufficient condition for the existence 
of a positive solution  $u \in L^{r}(\R^{n})$ ($0<r<\infty$) to this problem, which 
was open even in the case $p=2$.

Our version of Wolff's inequality  for intrinsic nonlinear potentials 
relies on a new characterization of  discrete Littlewood-Paley 
spaces $f^{p, q}(\sigma)$ defined in terms of characteristic functions of 
dyadic cubes in $\R^n$. 
\end{abstract}
\maketitle
\section{Introduction}\label{sect:intro} 

Let $ \mathcal{M}^{+}(\R^n)$ denote the class of all locally finite Borel measures 
on $\R^n$. For $1<r<\infty$ and 
$0<\alpha<\frac{n}{r}$, the Wolff potential, or, more precisely, Havin-Maz'ya-Wolff potential (see \cite{AH}, \cite{HM},  \cite{HW}, \cite{KM}, \cite{Maz}) $\W_{\alpha, r}\sigma$ of a measure 
$\sigma \in \mathcal{M}^{+}(\R^n)$ is defined by 
\begin{equation}\label{wolff}
\W_{\alpha, r}\sigma(x) = \int_{0}^{\infty} \left[ \frac{\sigma(B(x,\rho))}{s^{n-\alpha r}}\right]^{\frac{1}{r-1}}\; \frac{d\rho}{\rho}, \quad x \in \R^n,
\end{equation}
where $B(x,\rho) = \{ y \in \R^n : |x-y|< \rho \}$ is a ball centered at $x \in \R^n$ of radius $\rho >0$.

In the linear case $r=2$, the  potential $\W_{\alpha, r}\sigma$  reduces (up to a constant multiple) to the Riesz potential 
$\I_{2\alpha}\sigma$, where 
\[
\I_{\beta}\sigma(x)=\int_{\R^n} \frac{d \sigma (y)}{|x-y|^{n-\beta}}, \quad x \in \R^n, 
\]
is the Riesz potential of order $\beta\in (0, n)$.

In \cite{HW}, a useful dyadic version of $\W_{\alpha, r}$ was introduced:
\begin{equation}\label{wolff-d}
\W^d_{\alpha, r}\sigma(x) := \sum_{Q \in \mathcal{Q}}  \left[ \frac{\sigma(Q)}
{|Q|^{1-\frac{\alpha r}{n}}}\right]^{\frac{1}{r-1}}\; \chi_Q(x), \quad x \in \R^n, 
\end{equation}
where the sum is taken over all dyadic cubes $Q\in \mathcal{Q}$.

Clearly, 
\[
\W_{\alpha, r}\sigma(x)\ge c(\alpha, r, n) \, \W^d_{\alpha, r}\sigma(x). 
\]
The converse inequality can be recovered, as usual,  by replacing $\mathcal{Q}$ 
in \eqref{wolff-d} 
with a 
shifted dyadic lattice $\mathcal{Q}_t=\{Q+t\}$ ($t\in \R^n$), and then averaging over all 
$t\in B(0, 1)$ (see, for instance,   \cite{COV2}, \cite{COV3}, \cite{V}). 
Wolff's inequality obtained in \cite{HW} says that the $(\alpha, r)$-energy 
\begin{equation}\label{wolff-ineq}
\mathcal{E}_{\alpha, r}[\sigma] := \int_{\R^n} (\I_\alpha \sigma)^{r'} dx \le C(\alpha, r) \, 
\int_{\R^n} \W_{\alpha, r}\sigma \, d \sigma,
\end{equation}
where $1<r<\infty$ and $r'=\frac{r}{r-1}$. 
The converse inequality holds as well, since obviously, 
\[
\int_{\R^n} (\I_\alpha \sigma)^{r'} dx = \int_{\R^n} \V_{\alpha, r} \sigma \, d \sigma, 
\]
where  
\[
\V_{\alpha, r}\sigma(x)  := \I_\alpha [(\I_\alpha \sigma)]^{r'-1}(x), \quad x \in \R^n,
\]
is the Havin-Maz'ya potential introduced in \cite{HM}. It is easy to see (see  \cite{HM}, \cite{Maz}) that 
\[
\V_{\alpha, r}\sigma(x) \ge c (\alpha, r, n) \,  \W_{\alpha, r}(x). 
\]
Hence, Wolff's inequality demonstrates that 
\begin{equation}\label{wolff-in}
c   \, \int_{\R^n} \W_{\alpha, r}\sigma \, d \sigma \le \mathcal{E}_{\alpha, r}[\sigma] \le C   \, \int_{\R^n} \W_{\alpha, r}\sigma \, d \sigma, 
\end{equation}
where positive constants $c, C$ depend only on $\alpha, r, n$. 

One can also use dyadic potentials $\W^d_{\alpha, r}\sigma$ in place of 
$\W_{\alpha, r}\sigma$  in \eqref{wolff-in}, which yields the following 
discrete form of Wolff's inequality (\cite{HW}):
\begin{equation}\label{wolff-in-d}
 \mathcal{E}_{\alpha, r}[\sigma] \approx 
 \sum_{Q \in \mathcal{Q}}   \frac{ [\sigma(Q)]^{r'} }
{|Q|^{\frac{n-\alpha r}{(r-1)n}}},
  \end{equation}
where the constants of equivalence  depend only on $\alpha, r,$ and  $n$.

There are similar Wolff's inequalities for potentials $\W_{\alpha, p} \sigma$, since 
 for any $r>0$, $1<p<\infty$ and $0<\alpha<\frac{n}{p}$, we have 
 \[
 \Vert \W_{\alpha, p} \sigma\Vert^{p-1}_{L^r(\R^n)} \approx \Vert \I_{\alpha p} \sigma\Vert_{L^\frac{r}{p-1}(\R^n)}.  
 \]
 Consequently,
 \begin{equation}\label{wolff-w}
\int_{R^n}  (\W_{\alpha, p} \sigma)^r dx \approx \int_{R^n}  (\W^d_{\alpha, p} \sigma)^r dx\approx
 \sum_{Q \in \mathcal{Q}}   \frac{ [\sigma(Q)]^{\frac{r}{p-1}} }
{|Q|^{\frac{(n-\alpha p) r-n(p-1)}{(p-1)n}}},
  \end{equation}
where the constants of equivalence  depend only on $\alpha, p, r,$ and  $n$.

Several proofs of \eqref{wolff-in} and its variations are known; in particular, it can be deduced 
from a weighted norm inequality  of Muckenhoupt and Wheeden 
for fractional integrals \cite{MW} (see also \cite{AH}, \cite{HJ}, \cite{JPW}, \cite{V}). 
 However,  the original proof \cite{HW} 
via dyadic potentials $\W^d_{\alpha, r}\sigma$ is most direct, and useful 
in more general situations. (See, for instance, a two-weight version and its applications 
in \cite{COV3}, \cite{HV1}, \cite{HV2}.)

The following important result due to T.~Kilpel\"{a}inen and J.~Mal\'y 
\cite{KiMa} 
gives precise pointwise estimates of $p$-superharmonic solutions 
$u\ge 0$ to the equation 
\begin{equation}\label{p-lap}
\begin{cases}
-\Delta_{p} u = \sigma \quad \text{in} \;\; \mathbb{R}^n, \\ 
\liminf\limits_{\vert x \vert \rightarrow \infty} u(x) = 0,
\end{cases}
\end{equation}
in terms of potentials $\W_{1, p} \sigma$: \textit{Let $\sigma \in \mathcal{M}^{+}(\mathbb{R}^n)$ and $p>1$. Suppose 
$u$ is a $p$-superharmonic function in $\mathbb{R}^n$ 
satisfying \eqref{p-lap}. Then there exists a positive constant $K=K(n,p)$ such that}
\begin{equation}\label{k-m}
K^{-1}\W_{1,p} \sigma(x) \leq u(x) \leq K \W_{1,p}\sigma(x).
\end{equation} 
Moreover, such a solution to \eqref{p-lap} 
exists if and only if $1<p<n$, and $\W_{1,p} \sigma(x)<\infty$ for some  $x \in \R^n$, or equivalently, 
\begin{equation}\label{k-m-cond}
\int_{1}^{\infty} \left[ \frac{\sigma(B(0,\rho))}{s^{n-\alpha r}}\right]^{\frac{1}{r-1}}<\infty.
\end{equation} 
If \eqref{k-m-cond} holds, then $\W_{1,p} \sigma(x)<\infty$, $dx$-a.e. (and quasi-everywhere). 

Throughout  this paper, we use  $p$-superharmonic solutions, or equivalently, 
locally renormalized solutions to equations involving the $p$-Laplace operator. We refer to  \cite{HKM}, \cite{KKT} for the corresponding definitions and properties of such solutions.

Let us now consider the quasilinear elliptic problem 
\begin{equation}\label{eq:p-laplacian}
\begin{cases}
-\Delta_{p} u  = \sigma u^{q}, \quad u\ge 0  \quad \text{in} \;\; \R^n, \\
\liminf\limits_{x\rightarrow \infty}u(x) = 0,
\end{cases}
\end{equation}
in the sub-natural growth case $0<q< p-1$, where $ \sigma \in \M^{+}(\R^n)$. 
We assume here that $u\in L^{q}_{loc}(\Omega, d\sigma)$, so that the right-hand side of \eqref{eq:p-laplacian} is a Radon measure, and we can use $p$-superharmonic, 
or locally renormalized  solutions $u$, as in the case of \eqref{p-lap}. 

The existence  of solutions $u \in L^{\infty}(\R^n)$ 
to \eqref{eq:p-laplacian}  was characterized by Brezis and Kamin  \cite{BK} in the case $p=2$. 
They  also proved uniqueness of bounded solutions.  In fact, for all  $1<p<\infty$,
 a solution $u \in L^{\infty}(\R^n)$ 
to \eqref{eq:p-laplacian}  exists if and only if $\W_{1,p} \sigma  \in L^{\infty}(\R^n)$  (see \cite{CV3}). However, a similar problem for solutions $u \in L^{r}(\R^n)$ with $r<\infty$ 
turned out to be more complicated. Some sharp sufficient 
conditions for that were established recently in \cite{SV3} (see also \cite{SV1}, 
\cite{SV2} where finite energy solutions and their generalizations are treated). 

In this paper, we give a necessary and sufficient condition on $\sigma$, in terms of integrability of  nonlinear potentials, for the existence of a positive  solution 
$u \in L^{r}(\R^n)$, $0<r<\infty$, 
to problem \eqref{eq:p-laplacian}.

The following bilateral pointwise estimates of nontrivial (minimal) solutions $u$ to 
\eqref{eq:p-laplacian} in the  case $0<q<p-1$ were obtained in \cite{CV2}:
\begin{equation} \label{two-sided}
c^{-1} [(\W_{1, p} \sigma)^{\frac{p-1}{p-1-q}}+\mathbf{K}_{1, p, q}  \sigma] 
\le u \le c [(\W_{1, p} \sigma)^{\frac{p-1}{p-1-q}} +\mathbf{K}_{1, p, q}  \sigma],  
\end{equation}
where $c>0$ is a constant which  depends only on $p$, $q$, and $n$. 

Here $\mathbf{K}_{1, p, q}$ is the so-called \textit{intrinsic} Wolff potential 
associated with \eqref{eq:p-laplacian}, which was introduced in \cite{CV2}.  
It is defined in terms of the  localized weighted norm inequalities,    
\begin{equation} \label{weight-lap}
\left(\int_{B} |\varphi|^q \, d \sigma\right)^{\frac 1 q} \le \varkappa(B) \,  ||\Delta_p \varphi||^{\frac{1}{p-1}}_{L^1(\R^n)},  
\end{equation} 
for all test functions $\varphi$ such that  $-\Delta_p \varphi \ge 0$, $\displaystyle{\liminf_{x \to \infty}} \, \varphi(x)=0$. Here $\varkappa(B)$ denotes the least constant in \eqref{weight-lap} 
associated with the measure $\sigma_B=\sigma|_B$ restricted to a ball $B$.  
Then the intrinsic potential $\mathbf{K}_{1, p, q}$ is defined by 
\begin{equation} \label{potentialK}
\mathbf{K}_{1, p, q}  \sigma (x)  =  \int_0^{\infty} \left[\frac{ \varkappa(B(x, r))^{\frac{q(p-1)}{p-1-q}}}{r^{n- p}}\right]^{\frac{1}{p-1}}\frac{dr}{r}, \quad x \in \R^n.
\end{equation} 
As was shown in \cite{CV2},  $\mathbf{K}_{1, p, q}  \sigma \not\equiv + \infty$ if and only if 
\begin{equation} \label{suffcond1}
\int_1^{\infty} \left[\frac{ \varkappa(B(0, r))^{\frac{q(p-1)}{p-1-q}}}{r^{n- p}}\right]^{\frac{1}{p-1}}\frac{dr}{r} < \infty.  
\end{equation}

In a similar way, we define constants $ \varkappa(Q)$ for cubes $Q$ in place of $B$, 
and the dyadic potentials 
\begin{equation} \label{potentialK-d}
\mathbf{K}^d_{1, p, q}  \sigma (x)  =  \sum_{Q \in \mathcal{Q}} \left[\frac{ \varkappa(Q)^{\frac{q(p-1)}{p-1-q}}}{|Q|^{1- \frac{p}{n}}}\right]^{\frac{1}{p-1}}\chi_Q(x), \quad x \in \R^n.
\end{equation}

More general fractional potentials $\mathbf{K}_{\alpha, p, q}$, along with their dyadic analogues,  are defined in  
Sec. \ref{sect:wolff}.

Thus, a necessary 
and sufficient condition for the existence of a solution $u \in L^r(\R^n)$ 
to \eqref{eq:p-laplacian} is given by: 
\begin{equation} \label{cond-two-sided}
\W_{1, p}  \sigma \in L^{\frac{r(p-1)}{p-1-q}}(\R^n) \quad \textrm{and} \quad \mathbf{K}_{1, p, q}  \sigma \in L^r(\R^n).
\end{equation}

In fact, as we will show below, the first condition in  \eqref{cond-two-sided} 
is a consequence of the second one, despite differences in  pointwise behavior 
of $(\W_{1, p} \sigma)^{\frac{p-1}{p-1-q}}$ and $\mathbf{K}_{1, p, q}  \sigma$. 

 Moreover, we will simplify to 
some degree the second condition
 in \eqref{cond-two-sided} by proving an analogue of Wolff's inequality 
 \eqref{wolff-w} for potentials $\mathbf{K}_{1, p, q}  \sigma$.

Similar  results hold  for the fractional Laplace problem
\begin{equation} \label{eq:frac-laplacian}
\begin{cases}
\left(-\Delta \right)^{\alpha} u = \sigma u^{q}, \quad u \ge 0  \quad \text{in} \;\; \mathbb{R}^n, \\
\liminf\limits_{x \rightarrow \infty}u(x) = 0,
\end{cases}
\end{equation}
where $0<q<1$ and  $0< \alpha < \frac{n}{2}$. They are new even in the classical case $\alpha=1$,  or if $\sigma$ is a locally integrable function on $\R^n$.

\begin{Thm}\label{thm:main-p-laplacian}  
Let $1<p<n$, $0<q<p-1$, $\frac{n(p-1)}{n-p}<r<\infty$, and $\sigma \in \mathcal{M}^{+}(\R^n)$ with $\sigma \not\equiv 0$. Then the following conditions are equivalent.

(i) There exists a positive $p$-superharmonic (super) solution $u \in L^r(\R^{n})$ to \eqref{eq:p-laplacian}.

(ii) $\mathbf{K}_{1, p, q}  \sigma \in L^r(\R^n)$.

(iii) 
\begin{equation}\label{main-q} 
\sum_{Q \in \mathcal{Q}} \frac{[\varkappa(Q)]^{\frac{q r}{p-1-q}}}
{|Q|^{\frac{(n-p)r-n(p-1)}{n(p-1)}}}<\infty.
\end{equation}
Moreover, 
\begin{equation}\label{est-q}
\Vert \mathbf{K}_{1, p, q}  \sigma \Vert^r_{L^r(\R^n)} \approx \sum_{Q \in \mathcal{Q}} \frac{[\varkappa(Q)]^{\frac{q r}{p-1-q}}}
{|Q|^{\frac{(n-p)r-n(p-1)}{n(p-1)}}}, 
\end{equation}
where the constants of equivalence depend only  on $p, q, r$, and $n$. 
\end{Thm}

\begin{Rem}
It is easy to see that if $n \leq p < \infty$, or $1<p<n$ 
and $0<r\le \frac{n(p-1)}{n-p}$, then 
there is only a trivial nonnegative supersolution $u \in L^r(\R^n)$ to \eqref{eq:p-laplacian}. Simpler sufficient 
conditions for \eqref{main-q} in the case $1<p<n$, $\frac{n(p-1)}{n-p}<r<\infty$, 
are given in \cite{SV3}*{Theorem 1.1}. 
\end{Rem}

\begin{Rem}
A condition  equivalent  to \eqref{main-q}  can be stated in terms of 
$\varkappa(B)$ 
for balls $B=B(x, \rho)$ in place of dyadic cubes $Q$, 
\begin{equation}\label{main-b} 
\int_{\R^n} \int_0^\infty \frac{ [\varkappa(B(x, \rho))]^{\frac{q r}{p-1-q}}}{\rho^{\frac{(n-p) r}{p-1}}}
 \frac{d \rho}{\rho} dx<\infty.
\end{equation}
\end{Rem}

A necessary  (but generally not sufficient) condition 
for the existence of a nontrivial solution $u \in L^{r}(\R^n)$ to \eqref{eq:p-laplacian}   follows from \eqref{cond-two-sided},  
\begin{equation}\label{cond:dx-wolff-1} 
\W_{1,p} \sigma \in L^{\frac{r(p-1)}{p-1-q}}(\R^n). 
\end{equation}
By Wolff's inequality, \eqref{cond:dx-wolff-1}  is equivalent to the condition 
\[
\sum_{Q \in \mathcal{Q}} \frac{[\sigma(Q)]^{\frac{r}{p-1-q}}}
{|Q|^{\frac{(n- p) r-n(p-1-q)}{n(p-1-q)} }} <\infty. 
\]

\begin{Rem}\label{rem-main1}
Theorem \ref{thm:main-p-laplacian} 
holds for the $\A$-Laplacian in place of $\Delta_{p}$, under the standard structural  assumptions on $\mathcal{A}$ (see \cite{CV2}, \cite{HKM}, \cite{MZ}).
\end{Rem}

Our methods  are applicable to intrinsic nonlinear potentials of fractional order and nonlinear integral equations of the type 
\begin{equation}\label{eq:int2}
u = \W_{\alpha, p}(u^{q} d\sigma) \quad \text{in} \;\; \R^{n}.
\end{equation}
Here, a solution $u\ge 0$  is understood in the sense that
$u \in L^{q}_{loc}(\R^n, \sigma)$ satisfies \eqref{eq:int2}. In the special case $p=2$, 
this integral equation, namely  $u=\I_{2 \alpha} (u^q d \sigma)$,  is equivalent to the corresponding problem 
for the fractional Laplacian \eqref{eq:frac-laplacian}.

\begin{Thm}\label{thm:main2}
Let $0<q<p-1$, $1<p<\infty$, $0<\alpha<\frac{n}{p}$ and $\sigma \in \mathcal{M}^{+}(\R^n)$ with 
$\sigma \not\equiv 0$. Suppose that $\frac{n(p-1)}{n-\alpha p}<r<\infty$.  Then 
there exists a positive  solution $u \in L^r(\R^{n})$ to \eqref{eq:int2} if and only if 
$\mathbf{K}_{\alpha, p, q}  \sigma \in L^r(\R^n)$. Moreover,  
\begin{equation}\label{main-alpha} 
\Vert \mathbf{K}_{\alpha, p, q}  \sigma \Vert^r_{L^r(\R^n)} \approx  \sum_{Q \in \mathcal{Q}} \frac{[\kappa(Q)]^{\frac{q r}{p-1-q}}}
{|Q|^{\frac{(n-\alpha p)r-n(p-1)}{n(p-1)}}}, 
\end{equation}
where the constants of equivalence depend only  on $\alpha, p, q, r$, and $n$.

If $0<r\le \frac{n(p-1)}{n-\alpha p}$, then 
there is only a trivial supersolution to \eqref{eq:frac-laplacian}.
\end{Thm}

In \eqref{main-alpha}, we employ the localized embedding constants $\kappa(Q)$ 
associated with certain weighted norm inequalities for potentials 
$\W_{\alpha, p}$. They are used to define the intrinsic 
potentials $\mathbf{K}_{\alpha, p, q}  \sigma$ and their dyadic analogues in the same manner as constants $\varkappa(Q)$ 
in the case $\alpha=1$ above (see Sec. \ref{sect:wolff}). 

A simple necessary, but not sufficient, condition for \eqref{main-alpha}  is given by 
\[
\sum_{Q \in \mathcal{Q}} \frac{[\sigma(Q)]^{\frac{r}{p-1-q}}}
{|Q|^{\frac{(n- \alpha p) r-n(p-1-q)}{n(p-1-q)} }} <\infty,
\]

This paper is organized as follows. 
In Sec. \ref{sect:wolff}, we give definitions of  nonlinear potentials $\mathbf{K}_{\alpha, p, q}$ and discuss some of their properties. New expressions for 
norms of sequences in discrete Littlewood--Paley spaces $\mathbf{f}^{p, q} (\sigma)$ are discussed 
in Sec. \ref{Section 3}. They are used in Sec. \ref{Section 4}, where we prove Theorem \ref{thm:main-p-laplacian}   and Theorem \ref{thm:main2}.


\section{Nonlinear potentials}\label{sect:wolff}

Havin-Maz'ya potentials $\V_{\alpha, p}\sigma$ are known to satisfy the weak maximum (or boundedness) 
principle (see \cite{AH}*{Theorem 2.6.3}). A similar weak maximum principle holds 
for Wolff potentials: \textit{If $\sigma \in \mathcal{M}^{+}(\R^n)$, then} 
\begin{equation}\label{wolff-max}
\W_{\alpha, p}\sigma(x) \leq 2^{\frac{n-\alpha p}{p-1}} \, \sup \left \{\W_{\alpha, p}\sigma(y): \, y \in \textrm{supp}(\sigma)\right\}, \quad \forall x \in \R^n.
\end{equation}

Indeed, let $K=\textrm{supp}(\sigma)$. Suppose 
$x \not\in K$, and  $x_0\in K$ minimizes the distance from $x$ to $K$. 
Then, clearly, $B(x, r)\subset B(x_0, 2r)$, for any $r>0$. Consequently, 
\[
\W_{\alpha, p}\sigma(x)\le \int_{0}^{\infty} \left[  \frac{\sigma \left( B(x_0, 2r) \right)}{r^{n-\alpha p}}  \right]^{\frac{1}{p-1}}  
\frac{dr}{r}
 = 2^{\frac{n-\alpha p}{p-1}} 
\W_{\alpha, p}\sigma(x_0).
\]

Let $1<p<\infty$, 
$0<\alpha< \frac{n}{p}$, and $0<q<p-1$. Let $\sigma \in \mathcal{M}^{+}(\R^n)$. We denote by $\kappa$ the least constant in 
the weighted norm inequality 
\begin{equation} \label{kap-global}
||\W_{\alpha, p} \nu||_{L^q(\R^n, d\sigma)} \le \kappa  \, \nu(\R^n)^{\frac{1}{p-1}}, \quad \forall \nu \in \mathcal{M}^{+}(\R^n).  
\end{equation}
We will also need a localized version of \eqref{kap-global} for $\sigma_E=\sigma|_E$, where $E$ is 
a Borel subset   of $\R^n$, and $\kappa(E)$ is the least constant in 
\begin{equation} \label{kap-local}
||\W_{\alpha, p}  \nu||_{L^q(d\sigma_{E})} \le \kappa (E) \, \nu(\R^n)^{\frac{1}{p-1}}, \quad \forall \nu \in \mathcal{M}^{+}(\R^n). 
\end{equation}
In applications, it will be enough to use $\kappa(E)$ 
where $E=B$ is a dyadic cube $Q$, or a ball in $\R^n$. 

It is easy to see using estimates \eqref{k-m} that embedding constants 
$\kappa(B)$ in the case $\alpha=1$ are equivalent to the constants 
$\varkappa(B)$ in \eqref{weight-lap}.

We define the intrinsic potential of Wolff type in terms of $\kappa(B(x, s))$, the least constant in \eqref{kap-local} with $E=B(x, s)$: 
\begin{equation} \label{intrinsic-K}
\mathbf{K}_{\alpha, p, q} \sigma (x)  =  \int_0^{\infty} \left[\frac{ \kappa(B(x, s))^{\frac{q(p-1)}{p-1-q}}}{s^{n- \alpha p}}\right]^{\frac{1}{p-1}}\frac{ds}{s}, \quad x \in \R^n. 
\end{equation} 

It is easy to see that $\mathbf{K}_{\alpha, p, q} \sigma (x)  \not\equiv \infty$ if and only if 
\begin{equation} \label{inf}
 \int_a^{\infty} \left[\frac{ \kappa(B(0, s))^{\frac{q(p-1)}{p-1-q}}}{s^{n- \alpha p}}\right]^{\frac{1}{p-1}}\frac{ds}{s}< \infty,
\end{equation} 
for any (all) $a>0$.

As in the case of Wolff potentials $\W_{\alpha, p}$, sometimes a more convenient  dyadic version of $\mathbf{K}_{\alpha, p, q}$ 
is useful:
\begin{equation} \label{intrinsic-K-d}
\mathbf{K}^d_{\alpha, p, q} \sigma (x)  = \sum_{Q \in \mathcal{Q}}  \left[ \frac{ \kappa(Q)^{\frac{q(p-1)}{p-1-q}}} { |Q|^{1-\frac{\alpha p}{n}}}\right]^{\frac{1}{p-1}}\chi_Q(x), \quad x \in \R^n. 
\end{equation} 
Similarly to \eqref{inf}, 
 $\mathbf{K}_{\alpha, p, q}\not\equiv \infty$ if and only if 
\begin{equation} \label{inf-d}
 \sum_{R\supseteq P} \left[\frac{ \kappa(R)^{\frac{q(p-1)}{p-1-q}}}{|R|^{1- \frac{\alpha p}{n}}}\right]^{\frac{1}{p-1}}< \infty,
\end{equation} 
for  $P\in \mathcal{Q}$.


\section{Equivalent norms on discrete Littlewood-Paley spaces}\label{Section 3}

In this section, we give some new equivalent norms for 
discrete Littlewood-Paley spaces with respect to an arbitrary 
measure $\sigma\in \mathcal{M}^{+}(\R^n)$ (see \cite{CV}, \cite{FJ}), \cite{HV1}). 
In this paper, we will need them only in the case of Lebesgue measure, but a more general 
setup is useful in various applications in harmonic analysis and PDE (\cite{COV1}--\cite{COV3}, \cite{HV1}, \cite{HV2}). In particular, they give new characterizations 
of the discrete Carleson embedding theorem in the case $0<q<1<p$ (see 
Corollary \ref{cor-3} below).

 Let $\sigma \in \mathcal{M}^{+}(\R^n)$.  We use the notation 
 $|E|_\sigma=\sigma(E)$, for Borel sets 
$E\subset \R^n$; $|E|$ stands for Lebesgue measure of $E$. 

Let $\Lambda= (\lambda_Q)_{Q \in {\mathcal{Q}}}$ be a sequence 
of nonnegative reals.  We denote 
by $\mathcal D$ the collection of all dyadic cubes $Q\in \mathcal{Q}$ such that 
$|Q|_\sigma \not= 0$. 

For $0< \tau < \infty$, $0<r<\infty$, and $-\infty<q \le \infty$ ($q\not=0$), we set  
\begin{equation}\label{l1-3}
a_1(\Lambda) = 
\int_{\R^n} \Big ( \sum_{Q \in \mathcal D} \lambda_Q \chi_Q \Big )^{r} d \sigma,
\end{equation}
\begin{equation}\label{l2-3}
a_2(\Lambda) = \int_{\R^n}  \Big [ \sum_{R \in \mathcal D} \lambda_R \chi_R 
\Big ( \frac{1}{|R|_\sigma} 
\int_R ( \sum_{Q \subseteq R} \lambda_Q \chi_Q )^{q} d \sigma\Big )^{{\frac{1}{q}} (\frac{r}{\tau}-1)} 
\Big ]^{\tau} d \sigma. 
\end{equation}

For $0<r<1$, we set 
\begin{equation}\label{l3-3c} 
 a_3(\Lambda) = \sup   \Big\{ \sum_{R \in \mathcal D} \lambda_R^r \, \nu_R \Big\},
 \end{equation}
where the supremum   is taken over 
 all sequences  of nonnegative reals $\nu=(\nu_R)$ such that $\nu_Q=0$ if $|Q|_\sigma=0$, and 
\begin{equation}\label{l3-3carl}
\sum_{Q \subseteq P} \Big(\frac{\nu_Q}{|Q|_\sigma}\Big)^{\frac{1}{1-r}} |Q|_\sigma 
\le  |P|_\sigma,  \quad \forall P \in \mathcal D   
\Big\}.  
\end{equation} 
In other words, the supremum on the right-hand side of  \eqref{l3-3c}  is taken over 
 all Carleson sequences $\nu=(\nu_R)$ such that $||\nu||_{f^{\infty, \frac{1}{1-r}}_0(\sigma)} \le 1$. 
 
 For $-\infty<r<\infty$ and $-\infty <q <\infty$ ($q\not=0$), we set 
\begin{equation}\label{l4-3}
a_4(\Lambda) = \sum_{R \in \mathcal D} \lambda_R \, |R|_\sigma  \,
 \Big ( \frac{1}{|R|_\sigma} 
\int_R ( \sum_{Q \subseteq R} \lambda_Q \chi_Q )^{q} d \sigma
\Big )^{\frac{r-1}{q}}. 
\end{equation}
We observe that $a_4(\Lambda)$ coincides with $a_2(\Lambda)$  in the special case 
$\tau=1$, $r>0$.

The following duality lemma is known in the case $0<r<1$ (see \cite{CV}, \cite{HV1}).

\begin{Lem}\label{duality} Suppose $0<r<1$. There exists a positive constant $C$ depending only on $r$, $q$, and $\tau$ such that 
\begin{equation}\label{l2-3ca}
 \frac{1}{C} \, a_3(\Lambda) \le \, a_1(\Lambda) \le  C \, a_3(\Lambda).  
\end{equation}
 \end{Lem}

\begin{Thm}\label{theorem1-3}
Let  $0<q <r< \infty$, and $0< \tau <\infty$. 
Then there exists a positive constant $C$ depending only on $r$, $q$, and $\tau$ such that 
\begin{equation}\label{l2-3a}
 \frac{1}{C} \, a_2(\Lambda) \le \, a_1(\Lambda) \le  C \, a_2(\Lambda), 
\end{equation}
\end{Thm}

Theorem \ref{theorem1-3} is a consequence of the lemmas proved below.

The following corollary is immediate from Theorem \ref{theorem1-3}. 
\begin{Cor}\label{cor-3} Let $\sigma \in M^+(\R^n)$. Suppose $0<q<1<p<\infty$. Then the following statements are equivalent. 

(i) The ``one-weight'' inequality holds, 
$$
|| \sum_{Q \in \mathcal D} \lambda_Q \, \chi_Q \, \frac{1}{|Q|_\sigma} \int_Q |f| \, d \sigma||_{L^q(\sigma)} \le C \, ||f||_{L^p(\sigma)}, 
$$
for all $f \in L^p(\sigma)$.

(ii) $\sigma$ satisfies the condition
$$
|| \sum_{Q \in \mathcal D} \lambda_Q \, \chi_Q \, ||_{L^r(d \sigma)} < \infty, 
$$
where $r=\frac{pq}{p-q}$. 

(iii) $\sigma$ satisfies the condition
$$
|| \sum_{Q \in \mathcal D} \lambda_Q \, \chi_Q \, \Big(\frac{1}{|Q|_\sigma} \int_Q \rho_Q^q  d \sigma\Big)^{\frac{p'-1}{q}}||_{L^\tau(d \sigma)} < \infty,
$$
where $\tau = \frac{q(p-1)}{p-q}$, and 
\begin{equation}\label{rho}
\rho_Q(x) = \sum_{S \subseteq Q} \lambda_S \, \chi_S(x).
\end{equation}

\end{Cor}

We now prove a series of lemmas used in the proof of Theorem \ref{theorem1-3}. 
Some of them might be of independent interest.

\begin{Lem}\label{lemma-2} Let $\sigma \in M^+(\R^n)$, and let $-\infty<q\le \infty$.

(i) If  either $0< r < 1$ and  $-(1-r) \le q <\infty$, or $r \ge 1$ 
and $0<q<r<\infty$, 
then 
\begin{equation}\label{l3-3e} 
 a_4(\Lambda) \le  C \, a_1(\Lambda), 
\end{equation} 
where $C$ is a positive constant depending only on $r$.

(ii) If either $-\infty<q<r<1$, or $r\ge 1$ and $0<q<\infty$, then the converse inequality holds: 
\begin{equation}\label{l3-3f} 
 a_1 (\Lambda) \le  C \, a_4(\Lambda), 
\end{equation} 
where $C$ is a positive constant depending only on $r$ and $q$. 
\end{Lem}

\begin{proof} We first prove statement (i) in the case $0<r<1$. Since $q\ge -(1-r)$, we estimate using Jensen's inequality, 
\begin{equation*}
\Big ( \frac{1}{|R|_\sigma} 
\int_R ( \sum_{Q \subseteq R} \lambda_Q \chi_Q )^{q} d \sigma\Big )^{\frac{1}{q}} 
\ge \Big ( \frac{1}{|R|_\sigma} 
\int_R ( \sum_{Q \subseteq R} \lambda_Q \chi_Q )^{-(1-r)} d \sigma\Big )^{-\frac{1}{1-r}} 
\end{equation*}
Hence,  
\begin{equation*}
\begin{split}
a_4(\Lambda) & \le  \sum_{R \in \mathcal D} \lambda_R \, |R|_\sigma  
 \frac{1}{|R|_\sigma} 
\int_R ( \sum_{Q \subseteq R} \lambda_Q \chi_Q )^{q} d \sigma\\
& =  \int_R \sum_{R \in \mathcal D} \lambda_R \, \chi_R \, ( \sum_{Q \subseteq R} \lambda_Q \chi_Q )^{-(1-r)} d \sigma \le C({r}) \, a_1(\Lambda),  
\end{split}
\end{equation*}
where we used summation by parts in the last line. 

In the case $r \ge 1$ and $0<q<r<\infty$, we use the maximal function inequality 
 for the dyadic maximal operator $M_\sigma: \, L^{\frac{r}{q}}(\sigma) \to L^{\frac{r}{q}}(\sigma)$, $\frac{r}{q}>1$. Letting  $\phi = (\sum_{R \in \mathcal D} \lambda_R \, \chi_R)^q$, we estimate
 \begin{equation*}
\begin{split}
 a_4(\Lambda) &=   \int_{\R^n} \sum_{R \in \mathcal D} \lambda_R \, \chi_R \, 
 \Big ( \frac{1}{|R|_\sigma} 
\int_R ( \sum_{Q \subseteq R} \lambda_Q \chi_Q )^{q} d \sigma
\Big )^{\frac{r-1}{q}} d \sigma \\ &\le  \int_{\R^n} \phi^{\frac{1}{q}} (M_\sigma \phi)^{\frac{r-1}{q}} d \sigma \\ & \le \int_{\R^n}  (M_\sigma \phi)^{\frac{r}{q}} d \sigma \\ &\le C \, \int_{\R^n} \phi^{\frac{r}{q}}  d \sigma =C \, a_1(\Lambda).
  \end{split}
\end{equation*}

To prove statement (ii), by Jensen's inequality it suffices to assume $q>0$. The case $r=1$ 
is trivial. 
Suppose  $0<q<r<1$. 
Then 
 \begin{equation*}
\begin{split}
a_4(\Lambda)& =  \int_{\R^n} \sum_{R \in \mathcal D} \lambda_R \, \chi_R \, 
 \Big ( \frac{1}{|R|_\sigma} 
\int_R ( \sum_{Q \subseteq R} \lambda_Q \chi_Q )^{q} d \sigma
\Big )^{-\frac{1-r}{q}} d \sigma 
\\ & \ge   \int_{\R^n} \Big( \Big(\sum_{R \in \mathcal D} \lambda_R \, \chi_R\Big)^q \Big)^{\frac{1}{q}}
 \Big ( M_\sigma (\sum_{R \in \mathcal D} \lambda_R \, \chi_R)^q \Big)^{-\frac{1-r}{q}} d \sigma. 
  \end{split}
\end{equation*}

 Let 
 $\phi = (\sum_{R \in \mathcal D} \lambda_R \, \chi_R)^q$. To complete the proof of \eqref{l3-3f}, 
 it remains to show that, for $r>q$, 
 \begin{equation}\label{ineq-3}
 \int_{\R^n} \phi^{\frac{1}{q}} (M_\sigma \phi)^{-\frac{1-r}{q}} d \sigma \ge C \,  \int_{\R^n} \phi^{\frac{r}{q}} d \sigma=C \, a_1(\Lambda).
  \end{equation}
 The preceding inequality is proved using H\"older's inequality with exponents $\frac{1}{r}$ 
 and $\frac{1}{1-r}$, together with the maximal function inequality in $L^{\frac{r}{q}}(\sigma)$: 
  \begin{equation*}
\begin{split}
  \int_{\R^n} \phi^{\frac{r}{q}} d \sigma & =  \int_{\R^n} \phi^{\frac{r}{q}} (M_\sigma \phi)^{- r\frac{1-r}{q}}\cdot  (M_\sigma \phi)^{r\frac{1-r}{q}} d \sigma \\ & \le \Big(\int_{\R^n} \phi^{\frac{1}{q}} 
  (M_\sigma \phi)^{\frac{1-r}{q}} d \sigma \Big)^r   \Big(\int_{\R^n} 
  (M_\sigma \phi)^{\frac{r}{q}} d \sigma \Big)^{1-r} \\ & \le C  \Big(\int_{\R^n} \phi^{\frac{1}{q}} 
  (M_\sigma \phi)^{\frac{1-r}{q}} d \sigma \Big)^r   \Big(\int_{\R^n} 
  \phi^{\frac{r}{q}} d \sigma \Big)^{1-r},  
  \end{split}
\end{equation*}

 which yields \eqref{ineq-3}. Thus, $a_4(\Lambda) \ge C \, a_1(\Lambda)$. 
 
 We now consider the case $r> 1$. By Jensen's inequality, it suffices to consider 
 $q>0$ small enough, so without loss of generality we will assume  $0<q<\min(1, r-1)$.

 If $1 < r < \infty$, then by \eqref{rest1-3}, 
   \begin{equation}\label{rest2-3}
     a_1(\Lambda) \le r \, \sum_{R \in \mathcal D} \lambda_R \, |R|_\sigma 
   \Big( \frac{1}{|R|_\sigma} \int_R ( \sum_{Q \subseteq R} \lambda_Q \chi_Q )^{r-1} d \sigma\Big). 
   \end{equation}
   
      We will also need the elementary summation by parts inequality, for $r\ge 1$ (see 
      \cite{COV2}), 
        \begin{equation}
\begin{split}\label{rest1-3}
   \sum_{R \in \mathcal D} \lambda_R \, \chi_R &
   ( \sum_{Q \subseteq R}  \lambda_Q \, \chi_Q)^{r-1} \le ( \sum_{R \in \mathcal D} \lambda_R \, \chi_R )^r \\ & \le r \, \sum_{R \in \mathcal D} \lambda_R \, \chi_R 
   ( \sum_{Q \subseteq R}  \lambda_Q \, \chi_Q)^{r-1}. 
         \end{split}
\end{equation}

    We consider separately two subcases, $1+q < r \le 2$, and $r>2$. 
  
  Suppose first that $1 +q < r \le 2$. Then by H\"older's inequality with exponents  $t=\frac{1-q}{2-r} >1$ and 
  $t'=\frac{1-q}{r-1-q}$, 
     \begin{equation*}
\begin{split}
  & \frac{1}{|R|_\sigma} \int_R ( \sum_{Q \subseteq R} \lambda_Q \chi_Q )^{r-1} d \sigma \\ 
  & \le \Big(\frac{1}{|R|_\sigma} \int_R ( \sum_{Q \subseteq R} \lambda_Q \chi_Q )^{q} d \sigma
   \Big)^{\frac{2-r}{1-q}} \Big(\frac{1}{|R|_\sigma} \int_R  \sum_{Q \subseteq R} \lambda_Q \chi_Q 
   \, d \sigma
   \Big)^{\frac{r-1-q}{1-q}}\\ & = \Big(\frac{1}{|R|_\sigma} \int_R ( \sum_{Q \subseteq R} \lambda_Q \chi_Q )^{q} d \sigma
   \Big)^{\frac{2-r}{1-q}}
    \Big(\frac{1}{|R|_\sigma} \sum_{Q \subseteq R} \lambda_Q \, |Q|_\sigma
 \Big)^{\frac{r-1-q}{1-q}}.  
      \end{split}
\end{equation*}

    Substituting this estimate into \eqref{rest2-3}, we obtain 
      \begin{equation*}
\begin{split}
     a_1(\Lambda) &\le r \,  \sum_{R \in \mathcal D} \lambda_R \, |R|_\sigma  \Big(\frac{1}{|R|_\sigma}  \int_R ( \sum_{Q \subseteq R} \lambda_Q \chi_Q )^{q} d \sigma
   \Big)^{\frac{2-r}{1-q}}  \\ &\times 
     \Big(\frac{1}{|R|_\sigma} \sum_{Q \subseteq R} \lambda_Q \, |Q|_\sigma
 \Big)^{\frac{1-q}{r-1-q}}.
    \end{split}
\end{equation*}
  Using H\"older's inequality for sums with exponents 
    $s=\frac{(1-q)(r-1)}{(2-r)q}$ and $s'=\frac{(1-q)(r-1)}{r-1-q}$, we estimate 
         \begin{equation*}
\begin{split}
 a_1(\Lambda)  & \le r   \Big(\sum_{R \in \mathcal D} \lambda_R \, |R|_\sigma  \Big(\frac{1}{|R|_\sigma} \int_Q (\sum_{Q \subseteq R} \lambda_Q \, \chi_Q)^q d\sigma
 \Big)^{\frac{r-1}{q}}\Big)^{\frac{1}{s'}}\\ & \times  \Big(\sum_{R \in \mathcal D} \lambda_R \, |R|_\sigma  \Big(\frac{1}{|R|_\sigma} \sum_{Q \subseteq R} \lambda_Q \, |Q|_\sigma
 \Big)^{r-1}\Big)^{\frac{1}{s}}
  \\ & = r \,  a_4(\Lambda)^{\frac{1}{s'}} \Big(\sum_{R \in \mathcal D} \lambda_R \, |R|_\sigma  \Big(\frac{1}{|R|_\sigma} \sum_{Q \subseteq R} \lambda_Q \, |Q|_\sigma
 \Big)^{r-1}\Big)^{\frac{1}{s}}.
     \end{split}
\end{equation*}
  By the known estimate  for $r \ge 1$ (see \cite{COV2}), 
    \begin{equation}\label{rest4-3}
\frac{1}{C}  \,   a_1(\Lambda)  \le \sum_{R \in \mathcal D} \lambda_R \, |R|_\sigma  \Big(\frac{1}{|R|_\sigma} \sum_{Q \subseteq R} \lambda_Q \, |Q|_\sigma
 \Big)^{r-1} \le C  \,   a_1(\Lambda),  
   \end{equation}
   where  $C>0$ is a constant 
  which depends only on $r$.  
  Hence, 
  $$
       a_1(\Lambda) \le C \,   a_4(\Lambda)^{\frac{1}{s'}} \, a_1(\Lambda)^{\frac{1}{s}},
  $$
  which yields $a_1(\Lambda) \le C \, a_4(\Lambda)$. 
  
  It the second subcase $r>2$, assuming as above that $0<q<1$, we estimate by H\"older's inequality with exponents $t=\frac{r-1-q}{1-q}>1$ 
  and $t'=\frac{r-1-q}{r-2}$, 
        \begin{equation*}
\begin{split}
 & \frac{1}{|R|_\sigma}  \sum_{Q \subseteq R} \lambda_R \, |R|_\sigma 
 = \frac{1}{|R|_\sigma} \int_{R}  ( \sum_{Q \subseteq R} \lambda_R \chi_R)^q 
 \cdot  ( \sum_{Q \subseteq R} \lambda_R \chi_R )^{1-q} d \sigma 
 \\ & \le 
 \Big( \frac{1}{|R|_\sigma} \int_{R} ( \sum_{Q \subseteq R} \lambda_R \chi_R)^q  
 d \sigma \Big)^{\frac{1}{t'}} \Big( \frac{1}{|R|_\sigma} \int_{R} ( \sum_{Q \subseteq R} \lambda_R \chi_R )^{r-1}  
 d \sigma \Big)^{\frac{1}{t}}.  
  \end{split}
\end{equation*}
 By \eqref{rest4-3} and the preceding estimate,
      \begin{equation*}
\begin{split}
 a_1(\Lambda) & \le C  \sum_{R \in \mathcal D} \lambda_R \,  |R|_\sigma  \, 
  \Big(  \frac{1}{|R|_\sigma}   \sum_{Q \subseteq R} \lambda_Q \, |Q|_\sigma \Big)^{r-1} 
  d \sigma \\ 
  &\le C  \sum_{R \in \mathcal D} \lambda_R \,  |R|_\sigma  \,  \Big( \frac{1}{|R|_\sigma} \int_{R} ( \sum_{Q \subseteq R} \lambda_R \chi_R)^q  
 d \sigma \Big)^{\frac{r-1}{t'}} \\ & \times \Big( \frac{1}{|R|_\sigma} \int_{R} ( \sum_{Q \subseteq R} \lambda_R \chi_R )^{r-1}  
 d \sigma \Big)^{\frac{r-1}{t}}. 
    \end{split}
\end{equation*}
    Using now H\"older's inequality with exponents $s=\frac{r-1-q}{q(r-2)}>1$ 
  and $s'=\frac{r-1-q}{(r-1)(1-q)}$ for sums, so that $\frac{(r-1)s}{t'}=\frac{r-1}{q}$ and $\frac{(r-1)s'}{t}=1$, 
  we estimate, 
           \begin{equation*}
\begin{split}
 & a_1(\Lambda)  \le 
 C \Big( \sum_{R \in \mathcal D} \lambda_R \,  |R|_\sigma  \, \frac{1}{|R|_\sigma}  
 \int_R ( \sum_{Q \subseteq R} \lambda_Q \chi_Q )^{q} d \sigma)^{\frac{r-1}{q}}\Big)^{\frac{1}{s}}  \\ 
 & \times 
\Big( \sum_{R \in \mathcal D} \lambda_R \,  |R|_\sigma  \, \frac{1}{|R|_\sigma}  \int_R ( \sum_{Q \subseteq R} \lambda_Q \chi_Q )^{r-1} d \sigma\Big)^{\frac{1}{s'}} \le 
C \,  a_2(\Lambda)^{\frac{1}{s}} a_1(\Lambda)^{\frac{1}{s'}},
       \end{split}
\end{equation*}
  where we used \eqref{rest2-3} again in the last line.  This completes the proof of statement (ii). 
  \end{proof}

\begin{Lem}\label{lemma0} Let $\sigma \in M^+(\R^n)$, and let $0<q \le \infty$. 

(i) If either $0<\tau \le \min(r,1)$, or $q<r<\tau < \infty$, then 
\begin{equation}\label{l3-3a} 
 a_1(\Lambda) \le  C \, a_2(\Lambda), 
\end{equation} 
where $C$ is a positive constant depending only on $r$, $q$, and $\tau$.

(ii) If $\max (q, \tau) <r <\infty$, then the converse inequality holds: 
\begin{equation}\label{l3-3b} 
 a_2(\Lambda) \le  C \, a_1(\Lambda), 
\end{equation} 
where $C$ is a positive constant depending only on $r$, $q$, and $\tau$. 

\end{Lem}

\begin{proof} We first prove statement (i). Let $0<r<1$.   

Suppose $0<\tau\le r$. Set 
\begin{equation}\label{l3-4}
d_R = \frac{1}{|R|_\sigma} 
\int_R \Big ( \sum_{Q \subseteq R} \lambda_Q \chi_Q \Big )^{q} d \sigma. 
\end{equation} 
Suppose $\nu=(\nu_R)$ is a Carleson sequence such that $||\nu||_{f^{\infty, \frac{1}{1-r}}_0(\sigma)} \le 1$ as in \eqref{l3-3c}. 
Let $s= \frac{1-\tau}{1-r}\ge 1$. 
By H\"older's inequality with exponents $s$ and $s'$, 
  \begin{equation*}
\begin{split}
& \sum_{R \in \mathcal D} \lambda_R^r \,  \nu_R = 
 \sum_{R \in \mathcal D} \lambda_R^{\frac{\tau}{s}}\,  |R|_\sigma  \, \frac{\nu_R}{|R|_\sigma} \, 
 d_R^{\frac{\gamma}{s}} \cdot d_R^{-\frac{\gamma}{s}} \,  \lambda_R^{r-\frac{\tau}{s}}\\ & \le  \Big[
  \sum_{R \in \mathcal D} \lambda_R^{\tau} \, |R|_\sigma \,  \Big(\frac{\nu_R}{|R|_\sigma}\Big)^s  
 d_R^{\gamma} 
  \Big]^{\frac{1}{s}} 
   \Big[
  \sum_{R \in \mathcal D} \lambda_R^{(r- \frac{\tau}{s})s'} \, |R|_\sigma \, 
 d_R^{-\gamma(s'-1)}  
  \Big]^{\frac{1}{s'}}. 
      \end{split}
\end{equation*}

Note that $(r- \frac{\tau}{s})s'=1$, and $\gamma=\frac{r-\tau}{q}$, so that 
$$\frac{\gamma}{\tau}= \frac{1}{q}(\frac{r}{\tau}-1), \qquad \gamma (s'-1) = \frac{1-r}{q}.$$
Letting 
$$\mu_Q = \Big(\frac{\nu_Q}{|Q|_\sigma}\Big)^s \, |Q|_\sigma, \qquad Q \in \mathcal D, 
$$
we see that $||\mu||_{f^{\frac{1}{1-\tau}, \infty}_0(\sigma)}\le 1$, that is, 
$$
\sum_{Q \subseteq P}  \mu_Q^{\frac{1}{1-\tau}} |Q|_\sigma 
\le |P|_\sigma, \quad \forall \, P \in \mathcal D. 
$$
It follows from \eqref{l2-3ca}  and \eqref{l3-3c} with the exponent $\tau$ in place of $r$,  and $\mu_R$ in place 
of $\nu_R$, 
 \begin{equation*}
\begin{split}
 \sum_{R \in \mathcal D} \lambda_R^{\tau} \Big(\frac{\nu_R}{|R|_\sigma}\Big)^s  |R|_\sigma \, 
 d_R^{\gamma} & =  \sum_{R \in \mathcal D} \lambda_R^\tau \, 
 d_R^{\frac {\tau}{q} (\frac{r}{\tau}-1)} \, \mu_R
  \\ & \le \int_{\R^n} \Big( \sum_{R \in \mathcal D} \lambda_R \, 
 d_R^{\frac {1}{q} (\frac{r}{\tau}-1)} \, \chi_R \Big)^\tau d \sigma.
  \end{split}
\end{equation*}

Lemma \ref{lemma-2} (i)  yields  
\begin{equation*}
\begin{split}
\sum_{R \in \mathcal D} \lambda_R^{(r- \frac{\tau}{s})s'} 
 d_R^{-\gamma(s'-1)} |R|_\sigma 
   & =  \sum_{R \in \mathcal D}  \lambda_R \, |R|_\sigma \, d_R^{-\frac{1-r}{q}} 
   \\  & \le C \int_{\R^n} \left ( \sum_{Q \in \mathcal D} \lambda_Q \chi_Q \right )^{r} d \sigma.
   \end{split}
\end{equation*}

  Combining the preceding estimates, we obtain 
  \begin{equation*}
\begin{split}
  \int_{\R^n} \left ( \sum_{Q \in \mathcal D} \lambda_Q \, \chi_Q \right )^{r} d \sigma 
  & \le C \left( \int_{\R^n} \Big( \sum_{R \in \mathcal D} \lambda_R \, 
 d_R^{\frac {1}{q} (\frac{r}{\tau}-1)} \, \chi_R \,  \Big)^\tau d \sigma\right)^{\frac{1}{s}} \\
 & \times  \left(  \int_{\R^n} \Big( \sum_{R \in \mathcal D} \lambda_R \, \chi_R \Big)^r   d \sigma \right)^{\frac{1}{s'}},
      \end{split}
\end{equation*}
    which completes the proof of \eqref{l3-3a} in the case $0<\tau \le r <1$. 
    
    In the case $0<\tau<1$, $r \ge 1$, we set 
    \begin{equation}\label{l3-5}
a_R =\Big( \frac{1}{|R|_\sigma} 
\int_R \rho_R^q \, d \sigma\Big)^{\frac{r-\tau}{q \tau}}, 
\end{equation} 
where $\rho_R$ is defined by \eqref{rho}. 
Using summation by parts, we estimate 
    \begin{equation*}
\begin{split}
 a_2(\Lambda) & =  \int_{\R^n} \left ( \sum_{R \in \mathcal D} \lambda_R \, a_R \, \chi_R \right )^{\tau} d \sigma \ge  
 \sum_{R \in \mathcal D} \lambda_R \,  a_R \, |R|_\sigma \\ & \times
 \frac {1}{|R|_\sigma}  \int_{R} \Big( \sum_{Q \subseteq R} \lambda_Q \, a_Q \, \chi_Q \Big)^{\tau-1}   d \sigma.
       \end{split}
\end{equation*}
          We denote by $M_\sigma^R$ the dyadic maximal operator scaled to a cube  $R$: 
          \begin{equation}\label{l3-6}
          M_\sigma^R f(x) = \sup_{Q\in \mathcal D: \, x \in Q, \, Q \subseteq R}  \frac {1}{|Q|_\sigma}  \int_{Q} 
          |f| \, d \sigma, \quad x \in R.
               \end{equation} 
     
     Clearly, 
     $$
     \sum_{Q \subseteq R} \lambda_Q \, a_Q \, \chi_Q \le 
    \rho_R \,  (M_\sigma^R \rho_R^q)^{\frac{r-\tau}{q \tau}}.
    $$
    Hence, by Jensen's inequality and the maximal inequality for $M_\sigma^R$, 
         \begin{equation*}
\begin{split}
    & \frac {1}{|R|_\sigma}  \int_{R} \Big( \sum_{Q \subseteq R} \lambda_Q \, a_Q \, \chi_Q \Big)^{\tau-1}   d \sigma  \ge  C  \frac {1}{|R|_\sigma}   \int_{R} (M_\sigma^R \rho_R^q)^{-\frac{r}{q \tau}} d \sigma \\ & \ge C 
    \Big( \frac {1}{|R|_\sigma}   \int_{R} (M_\sigma^R \rho_R^q)^{\epsilon}
     d \sigma\Big)^{-\frac{r(1-\tau)}{q \tau \epsilon}} 
      \ge  C \,   \Big( \frac {1}{|R|_\sigma}   \int_{R} \rho_R^q 
     d \sigma\Big)^{-\frac{r(1-\tau)}{q \tau}}, 
         \end{split}
\end{equation*} 
      where in the last line we used Kolmogorov's maximal inequality for $M^R_\sigma: \, L^1(\omega) \to L^{\epsilon} (\omega)$, $0<\epsilon<1$,  with the probability measure $d \omega=\frac {1}{|R|_\sigma} \chi_R d \sigma$, applied to $f=\rho_R^q$ .
     
     Consequently, 
       \begin{equation*}
\begin{split}
   a_2(\Lambda) & \ge C  \sum_{R \in \mathcal D} \lambda_R \,  a_R \, |R|_\sigma 
    \Big( \frac {1}{|R|_\sigma}   \int_{R} \rho_R^q 
     d \sigma\Big)^{-\frac{r(1-\tau)}{q \tau}} \\ & = C  \sum_{R \in \mathcal D}
      \lambda_R \, |R|_\sigma 
    \Big( \frac {1}{|R|_\sigma}   \int_{R} \rho_R^q 
     d \sigma\Big)^{\frac{r-1}{q}} = C a_4(\Lambda).  
     \end{split}
\end{equation*} 
    By Lemma \ref{lemma-2}, we have $a_4(\Lambda) \ge C a_1(\Lambda)$, which completes the proof 
    of \eqref{l3-3a} in the case $0<\tau<1$, $r \ge 1$.

    Suppose now that $\tau>r>q$. Then 
         \begin{equation*}
\begin{split}
   &   \int_{\R^n} \Big( \sum_{R \in \mathcal D} \lambda_R \, 
 d_R^{\frac {1}{q} (\frac{r}{\tau}-1)} \, \chi_R \,  \Big)^\tau d \sigma\\
 & \ge      \int_{\R^n} \Big( \sum_{R \in \mathcal D} \lambda_R  \, \chi_R  \Big)^\tau 
 \Big(M_\sigma \Big ( \sum_{R \in \mathcal D} \lambda_R  \, \chi_R\Big)^q  \Big)^{\frac{r-\tau}{q}} d \sigma \\ & \ge  C \int_{\R^n} \Big ( \sum_{Q \in \mathcal D} \lambda_Q \, \chi_Q \Big )^{r} d \sigma.
         \end{split}
\end{equation*} 
      The last inequality follows, as in the proof of Lemma \ref{lemma-2}, by letting $\phi = \Big ( \sum_{Q \in \mathcal D} \lambda_Q \, \chi_Q \Big )^{q}$, and applying H\"older's inequality with exponents $\frac{\tau}{r}$ and 
      $(\frac{\tau}{r})^{\prime}$, together with the maximal function inequality in $L^{\frac{r}{q}}(d \sigma)$ for 
      $\frac{r}{q}>1$: 
           \begin{equation*}
\begin{split}
      \int_{\R^n} \phi^{\frac{r}{q}} d \sigma & =       \int_{\R^n} \phi^{\frac{r}{q}} \, 
      (M_\sigma \phi)^{\frac{r-\tau}{q} \frac{r}{\tau}} \cdot   (M_\sigma \phi)^{\frac{\tau-r}{q} \frac{r}{\tau}} \, d \sigma \\ & \le \Big(    \int_{\R^n}  \phi^{\frac{\tau}{q}} \, 
      (M_\sigma \phi)^{\frac{r-\tau}{q}} \, d \sigma
      \Big)^{\frac{r}{\tau}}   \Big(    \int_{\R^n}   \, 
      (M_\sigma \phi)^{\frac{r}{q}} \, d \sigma
      \Big)^{1-\frac{r}{\tau}}\\ 
     & \le C \Big(    \int_{\R^n}  \phi^{\frac{\tau}{q}} \, 
      (M_\sigma \phi)^{\frac{r-\tau}{q}} \, d \sigma
      \Big)^{\frac{r}{\tau}}   \Big(    \int_{\R^n}  \phi^{\frac{r}{q}} \, d \sigma
      \Big)^{1-\frac{r}{\tau}}.
  \end{split}
\end{equation*}  
      This proves the inequality $a_1(\Lambda) \le C \,    a_2(\Lambda)$. 
      
      The converse inequality for $\max (q, \tau) <r <\infty$ is immediate from the maximal 
      function inequality in $L^{\frac{r}{q}}(\sigma)$: if $\phi = \Big ( \sum_{Q \in \mathcal D} \lambda_Q \, \chi_Q \Big )^{q}$, then 
         \begin{equation*}
\begin{split}
      a_2(\Lambda) & \le  \int_{\R^n} \phi^{\frac{\tau}{q}} 
      (M_\sigma \phi)^{\frac{r-\tau}{q}} d \sigma \le \int_{\R^n} 
      (M_\sigma \phi)^{\frac{r}{q}} d \sigma \\  & \le C \, \int_{\R^n} \phi^{\frac{r}{q}} d \sigma = C \, a_1(\Lambda). 
        \end{split}
\end{equation*}  
      \end{proof}
      

\begin{Lem}\label{lemma-3} Let $\sigma \in M^+(\R^n)$, and let $0< q \le \infty$.

(i) If $1 \le \tau <r$, then 
\begin{equation}\label{l-1-a} 
 a_1(\Lambda) \le  C \, a_2(\Lambda). 
\end{equation}

(ii) If $0<r< \tau$, then 
\begin{equation}\label{l-1-b} 
 a_2(\Lambda) \le  C  \, a_1(\Lambda), 
\end{equation} 
\end{Lem} 

\begin{proof} Since $r \ge 1$, by Lemma \ref{lemma-2} (ii), for every  $s>0$, 
\begin{equation}\label{l/1} 
 a_1(\Lambda) \le  C \, \, \sum_{R \in \mathcal D} \lambda_R \, |R|_\sigma  \Big( \frac{1}{|R|_\sigma} \int_R ( \sum_{Q \subseteq R} \lambda_Q \chi_Q )^{s} d \sigma \Big)^{\frac{r-1}{s}}. 
\end{equation} 
On the other hand, letting 
\begin{equation}\label{l/2}
a_R = \Big( \frac{1}{|R|_\sigma} \int_R ( \sum_{Q \subseteq R} \lambda_Q \chi_Q )^{q} d \sigma \Big)^{\frac{r-\tau}{q \tau}}, 
\end{equation}
and applying Lemma \ref{lemma-2} (i) with $r = \tau\ge 1$, $q=1$, and $\lambda_R \, a_R$ in place of 
$\lambda_R$, we obtain 
  \begin{equation*}
\begin{split}
     a_2(\Lambda) & = \int_{\R^n}  (   \sum_{R \in \mathcal D} \lambda_R \, a_R \, \chi_R)^{\tau} d \sigma
      \\ & \ge C \,  \sum_{R \in \mathcal D} \lambda_R \, a_R \, |R|_\sigma \Big( 
     \frac{1}{|R|_\sigma} \int_R ( \sum_{Q \subseteq R} \lambda_Q \, a_Q \, \chi_Q ) d \sigma  
      \Big)^{\tau-1}.
        \end{split}
\end{equation*}  
      
      By Jensen's inequality, it suffices to prove \eqref{l3-3a} for $q$ small, so that 
we may assume without loss of generality $r> \tau(1+q)$. Then 
$\gamma=\frac{1}{q}(\frac{r}{\tau}-1)>1$. 

Notice that 
  \begin{equation}
\begin{split}\label{l/5}  
a_2(\Lambda) & = \int_{\R^n}  (   \sum_{R \in \mathcal D} \lambda_R \, a_R \, \chi_R)^{\tau} d \sigma
      \\ & \ge C \,  \sum_{R \in \mathcal D} \lambda_R \, a_R \, |R|_\sigma \Big( 
     \frac{1}{|R|_\sigma} \int_R ( \sum_{Q \subseteq R} \lambda_Q \, a_Q \, \chi_Q ) d \sigma  
      \Big)^{\tau-1}.
       \end{split}
\end{equation}  

Next, we estimate using Jensen's inequality for sums, 
 \begin{equation}
\begin{split}\label{l/6}      
  & \int_R ( \sum_{Q \subseteq R} \lambda_Q \, a_Q \, \chi_Q) d \sigma =     \sum_{Q \subseteq R} \lambda_Q \,  a_Q \, |Q|_\sigma \\ &=   \sum_{Q \subseteq R} \lambda_Q \,  |Q|_\sigma \, 
  \Big(  \frac{1}{|Q|_\sigma}  \int_Q ( \sum_{S \subseteq Q} \lambda_S \,  \chi_S )^{q} d \sigma
  \Big)^{\gamma}\\ & \ge \Big( 
  \sum_{Q \subseteq R} \lambda_Q \,  |Q|_\sigma \, 
   \frac{1}{|Q|_\sigma}  \int_Q ( \sum_{S \subseteq Q} \lambda_S \,  \chi_S )^{q} d \sigma
   \Big )^{\gamma} \Big( \sum_{Q \subseteq R} \lambda_Q \,  |Q|_\sigma  \Big )^{1-\gamma}. 
    \end{split}
\end{equation}  
 
We simplify using summation by parts,
 \begin{equation*}
\begin{split}
 \sum_{Q \subseteq R} \lambda_Q \, 
    \int_Q ( \sum_{S \subseteq Q} \lambda_S \,  \chi_S )^{q} d \sigma 
    & =   \int_R \sum_{Q \subseteq R} \lambda_Q \, \chi_Q \, 
     ( \sum_{S \subseteq Q} \lambda_S \,  \chi_S )^{q} d \sigma \\ &
     \ge C \,  \int_R ( \sum_{Q \subseteq R} \lambda_Q \,  \chi_Q )^{1+q} d \sigma.
      \end{split}
\end{equation*}  
      Hence, combining the preceding estimates and using the interpolation inequality, 
         \begin{equation*}
\begin{split}
& \Big( \int_R ( \sum_{Q \subseteq R} \lambda_Q \, 
 \chi_Q )^{1+q} d \sigma \Big)^{\gamma(\tau-1)} \Big( \int_R ( \sum_{Q \subseteq R} \lambda_Q \,  \chi_Q )^{q} d \sigma \Big)^{\gamma} \\ & \ge 
   \int_R ( \sum_{Q \subseteq R} \lambda_Q \,  \chi_Q ) \, d \sigma \Big)^{r-\tau + 
   \gamma(\tau-1)},
          \end{split}
\end{equation*}   
      we obtain
       \begin{equation*}
\begin{split}     
       a_2(\Lambda) & \ge C \,  \sum_{R \in \mathcal D} \lambda_R \, |R|_\sigma  \, a_R \, 
      \Big(  \frac{1}{|R|_\sigma}   \int_R ( \sum_{Q \subseteq R} \lambda_Q \,  \chi_Q )^{1+q} d \sigma \Big)^{\gamma(\tau-1)} \\ & \times
       \Big(  \frac{1}{|R|_\sigma}   \int_R ( \sum_{Q \subseteq R} \lambda_Q \,  \chi_Q ) \, d \sigma \Big)^{(1-\gamma)(\tau-1)}\\
      & \ge C \,  
      \sum_{R \in \mathcal D} \lambda_R \, |R|_\sigma  \, \Big(  \frac{1}{|R|_\sigma}   \int_R ( \sum_{Q \subseteq R} \lambda_Q \,  \chi_Q ) \, d \sigma \Big)^{r-1} \\ &= C \, a_4(\Lambda) \ge C \, a_1(\Lambda),
                   \end{split}
\end{equation*}    
            where in the last line we used Lemma \ref{lemma-2} (ii) with $q=1$ in the expression 
            for $a_4(\Lambda)$, and $r\ge 1$.  
            This completes the proof of statement~(i) of Lemma \ref{lemma-3}. 
            
            To prove statement (ii), we may assume without loss of generality that $q$ is small enough; in particular, $0<q \le \frac{r}{\tau}$, where $r<\tau$. Let $\gamma=-\frac{1}{q} (\frac{r}{\tau}-1)>0$. 
            
            Consider first 
            the case $\tau \ge 1$. By Lemma \ref{lemma-2} (ii), with $\tau$ in place of $r$ 
            and $q=1$, we estimate 
             \begin{equation*} 
       a_2(\Lambda) \le C \,  \sum_{R \in \mathcal D} \lambda_R \, |R|_\sigma  \, a_R \, 
            \Big(  \frac{1}{|R|_\sigma}  \sum_{Q \subseteq R} \lambda_Q \,  |Q|_\sigma \, a_Q \Big)^{\tau-1}.
            \end{equation*} 
  Note that 
       \begin{equation} \label{l/7}
       \begin{split}
  a_Q & =  \Big(  \frac{1}{|Q|_\sigma}   \int_Q ( \sum_{S \subseteq Q} \lambda_S \,  \chi_S )^q \, d \sigma \Big)^{-\gamma} \\ & \le \Big(  \frac{1}{|Q|_\sigma}   \int_Q ( \sum_{S \subseteq Q} \lambda_S \,  \chi_S )^{-\epsilon} \, d \sigma \Big)^{\frac{q\gamma}{\epsilon}}, 
  \end{split}
       \end{equation}  
  for any $\epsilon>0$.       Using the preceding inequality with $\epsilon=1-\frac{r}{\tau}$, 
  we estimate 
               \begin{equation*}
\begin{split}    
           \frac{1}{|R|_\sigma}    \sum_{Q \subseteq R} \lambda_Q \,  |Q|_\sigma \, a_Q  & = 
           \frac{1}{|R|_\sigma}     \sum_{Q \subseteq R} \lambda_Q \,  |Q|_\sigma \, 
             \Big(  \frac{1}{|Q|_\sigma}   \int_Q ( \sum_{S \subseteq Q} \lambda_S \,  \chi_S )^q \, d \sigma \Big)^{-\gamma}\\ & \le 
               \frac{1}{|R|_\sigma} \sum_{Q \subseteq R} \lambda_Q \,  |Q|_\sigma \, 
               \frac{1}{|Q|_\sigma}   \int_Q 
             ( \sum_{S \subseteq Q} \lambda_S \,  \chi_S )^{-(1-\frac{r}{\tau})} \, d \sigma 
          \\ & =   \frac{1}{|R|_\sigma}   \int_R \sum_{Q \subseteq R} \lambda_Q \,  \chi_Q \, ( \sum_{S \subseteq Q} \lambda_S \,  \chi_S )^{-(1-\frac{r}{\tau})}  \, d \sigma
            \\
    &  \le  C \, \frac{1}{|R|_\sigma}  
       \int_R  ( \sum_{Q \subseteq R} \lambda_Q \, \chi_Q)^{\frac{r}{\tau}} d \sigma, 
        \end{split}
\end{equation*}   
       where we used summation by parts in the last line. 
       
       Since we are assuming that $0<q\le \frac{r}{\tau}<1$, it follows by Jensen's inequality, 
           \begin{equation*}
\begin{split}     
a_2(\Lambda) & \le C \,   \sum_{R \in \mathcal D} \lambda_R \, |R|_\sigma  \, a_R \, 
\Big (   \frac{1}{|R|_\sigma}   \int_R  ( \sum_{Q \subseteq R} \lambda_Q \,  \chi_Q )^{\frac{r}{\tau}} d \sigma
\Big)^{\tau-1} \\
           &  \le C \, \sum_{R \in \mathcal D} \lambda_R \, |R|_\sigma  \, 
            \Big ( \frac{1}{|R|_\sigma}   \int_R  ( \sum_{Q \subseteq R} \lambda_Q \,  \chi_Q )^{\frac{r}{\tau}} d \sigma
\Big)^{\frac{\tau}{r}(r-1)} \le C a_1(\Lambda),  
                  \end{split}
\end{equation*}    
        where in the last inequality we used Lemma \ref{lemma-2} (ii) with $\frac{r}{\tau}$ in place of $q$.  
            
            In the case $r<\tau<1$, we can assume again that $q$ is small enough; in particular, 
            $0<q<r$. 
        Using Lemma \ref{lemma-2} (ii) again for the sequence $\lambda_R \,  a_R$,  
        with $r_1=\tau$ in place of $r$, and $q_1=s$ in place of $q$ where $0<s<\tau$,  in the expression 
            for $a_4(\Lambda)$,  we have 
             \begin{equation*} 
       a_2(\Lambda) \le C \,  \sum_{R \in \mathcal D} \lambda_R \, |R|_\sigma  \, a_R \, 
            \Big(  \frac{1}{|R|_\sigma}  \int_R (\sum_{Q \subseteq R} \lambda_Q \, a_Q \, \chi_Q)^s d \sigma\Big)^{\frac{\tau-1}{s}}.  
       \end{equation*}          
            
           Let $\phi_R = (\sum_{Q \subseteq R} \lambda_Q \, \chi_Q)^{q}$. Then 
           $$a_Q \ge \Big(M^R_\sigma \phi_R (y)\Big)^{-\gamma}, \quad y\in R, \, \, Q \subseteq R,$$
          where $M^R_\sigma$ is the localized maximal  function \eqref{l3-6}. Hence, 
           $$
            \frac{1}{|R|_\sigma}  \int_R (\sum_{Q \subseteq R} \lambda_Q \, a_Q \, \chi_Q)^s d \sigma \ge    \frac{1}{|R|_\sigma}  \int_R \phi_R^{\frac{s}{q}} (M^R_\sigma \phi_R)^{-\frac{\gamma s}{q}} d \sigma.
           $$
         By H\"older's inequality and the maximal function inequality (see \eqref{ineq-3}), 
        \begin{equation*}      
 \int_R \phi_R^{\frac{s}{q}} (M^R_\sigma \phi_R)^{-\frac{\gamma s}{q}} d \sigma = 
 \int_R \Big(\frac{\phi}{M^R_\sigma \phi_R}\Big)^{p} (M^R_\sigma \phi_R)^{\delta p}\\
 \ge C \,  \int_R \phi_R^{\delta p},
          \end{equation*}  
          where $ p=\frac{s}{q}>1$     and $\delta = \frac{r}{\tau}<1$. Hence,
           \begin{equation*}
           \Big(  \frac{1}{|R|_\sigma}  \int_R (\sum_{Q \subseteq R} \lambda_Q \, a_Q \, \chi_Q)^s d \sigma\Big)^{\frac{\tau-1}{s}} \le C \, 
           \Big( \frac{1}{|R|_\sigma}  \int_R (\sum_{Q \subseteq R} \lambda_Q \, \chi_Q)^{\frac{r s}{\tau}} d \sigma
           \Big)^{\frac{\tau-1}{s}}.
           \end{equation*}  
         
         Assuming without loss of generality that $q$ is small enough, 
         so that $0<q\le \frac{r s}{\tau}$, 
         and using Jensen's inequality we estimate 
           \begin{equation*}
           \Big( \frac{1}{|R|_\sigma}  \int_R (\sum_{Q \subseteq R} \lambda_Q \, \chi_Q)^{\frac{r s}{\tau}} d \sigma
           \Big)^{- \frac{\gamma \tau}{rs}} \le a_R. 
                \end{equation*} 
         Consequently, for $r<\tau<1$, 
            \begin{equation*}
\begin{split}       
       a_2(\Lambda)     
      & \le C \,  
      \sum_{R \in \mathcal D} \lambda_R \, |R|_\sigma  \, a_R \,     \Big( \frac{1}{|R|_\sigma}  \int_R (\sum_{Q \subseteq R} \lambda_Q \, \chi_Q)^{\frac{r s}{\tau}} d \sigma
           \Big)^{\frac{\tau-1}{s}}    
       \\ & \le  C \, \sum_{R \in \mathcal D} \lambda_R \, |R|_\sigma  \, 
        \Big(   \frac{1}{|R|_\sigma} 
        \int_R ( \sum_{Q \subseteq R} \lambda_Q \,  \chi_Q )^{q} \, d \sigma \Big)^{\frac{r-1} q}
         \\ & = C \, a_4(\Lambda) \le C \, a_1(\Lambda), \qquad \qquad
             \end{split}
\end{equation*}    
which proves statement (ii).\end{proof}

The following corollary, which is merely a combination of Lemma \ref{lemma0} and Lemma \ref{lemma-3}, yields Theorem \ref{theorem1-3}. 

\begin{Cor}\label{corollary-3} Let $\sigma \in M^+(\R^n)$, and let $0< q \le \infty$.

(i) If either $0 \le \tau \le r$, or $r<q<\tau$, 
then 
\begin{equation*}
 a_1(\Lambda) \le  C \, a_2(\Lambda).  
\end{equation*} 

(ii) If $0<q< r$ and $0< \tau \le r$, then 
\begin{equation*} 
 a_2(\Lambda) \le  C \, a_1(\Lambda).  
\end{equation*} 
\end{Cor} 

\begin{Rem} Statement (i) of Corollary \ref{corollary-3}  
 fails if $0<r< \min (\tau, q)$; statement (ii)  fails if $q>r>\tau$.
 \end{Rem}


\section{Proofs of Theorems \ref{thm:main-p-laplacian} and \ref{thm:main2}}\label{Section 4} 

It is shown in  \cite{CV2} that \eqref{eq:p-laplacian} has a positive (super) solution 
if and only if the same is true for  \eqref{eq:int2} in the case $\alpha=1$. Moreover, 
the conditions in Theorems \ref{thm:main-p-laplacian} and \ref{thm:main2} are equivalent, since one can use embedding constants $\kappa(B)$ in place 
of $\varkappa(B)$ if $\alpha=1$ (see Sec. \ref{sect:wolff}). Thus, it suffices to prove only Theorem \ref{thm:main2}. 

Let $u \in L^q_{{\rm loc}}(\sigma)$ ($u\ge 0$) be a solution to  \eqref{eq:int2}. 
In \cite{CV2}, the following analogue of the bilateral pointwise estimates 
\eqref{two-sided} was obtained for nontrivial (minimal) solutions $u$ to \eqref{eq:int2}  in the  case $0<q<p-1$:
\begin{equation} \label{two-sided-frac}
c^{-1} [(\W_{\alpha, p} \sigma)^{\frac{p-1}{p-1-q}}+\mathbf{K}_{\alpha, p, q}  \sigma] 
\le u \le c [(\W_{\alpha, p} \sigma)^{\frac{p-1}{p-1-q}} +\mathbf{K}_{\alpha, p, q}  \sigma],  
\end{equation}
where $c>0$ is a constant which  depends only on $\alpha$, $p$, $q$, and $n$. 
Moreover  a nontrivial (super) solution exists if and only if both $\W_{\alpha, p} \sigma\not\equiv \infty$ and $\mathbf{K}_{\alpha, p, q}\not\equiv \infty$. 

It follows that 
$u \in L^r(\R^n)$ ($r>0$) exists if and only the following analogue of 
\eqref{cond-two-sided} holds: 
\begin{equation}\label{cond-two-sided-fr} 
 \mathbf{K}_{\alpha, p, q}  \sigma \in L^r(\R^n), \quad \W_{\alpha,p} \sigma \in L^{\frac{r(p-1)}{p-1-q}}(\R^n). 
\end{equation}
The first condition here   actually follows from the second one, both in \eqref{cond-two-sided} (in the case $\alpha=1$), and in 
\eqref{cond-two-sided-fr} 
that is, 
\begin{equation} \label{cond-two}
  \mathbf{K}_{\alpha, p, q}  \sigma \in L^r(\R^n) \Longrightarrow  \W_{\alpha, p}  \sigma \in L^{\frac{r(p-1)}{p-1-q}}(\R^n).
\end{equation}

Indeed, suppose that $\mathbf{K}_{\alpha, p, q}  \sigma \in L^r(\R^n) $. 
Using the following trivial estimate for balls $B=B(x, \rho)$, 
\begin{equation}\label{triv}
\sigma(B) |B|^{-\frac{n-\alpha p}{n(p-1)}}\le C \, [\kappa(B)]^q, 
\end{equation}
we see that 
\[
 \mathbf{K}_{\alpha, p, q}  \sigma(x) \ge C \, \int_0^\infty
 \Big[ \frac{\sigma(B(x, \rho))}{\rho^{n-\alpha p}}\Big]^{\frac{1}{p-1-q}} \frac{d \rho}{\rho}. 
\]
Hence,
\[
\int_0^\infty \Big[ \frac{\sigma(B(x, \rho))}{\rho^{n-\alpha p}}\Big]^{\frac{1}{p-1-q}} \frac{d \rho}{\rho} \in L^r(\R^n). 
\]

Estimates in  \cite{HJ}, \cite{JPW} yield that the preceding condition   is equivalent to $\W_{\alpha, p}  \sigma \in L^{\frac{r(p-1)}{p-1-q}}(\R^n)$. This proves \eqref{cond-two}.

It remains to show that $ \mathbf{K}_{\alpha, p, q}  \sigma \in L^r(\R^n)$ is equivalent to \eqref{main-q}. 

Suppose that $ \mathbf{K}_{\alpha, p, q}  \sigma \in L^r(\R^n)$. Then by 
\eqref{cond-two}, there exists a nontrivial (super) solution $u \in L^r(\R^n)$ to  either \eqref{eq:p-laplacian} or \eqref{eq:int2}. We set $d \omega=u^q d\sigma$. Then 
$\W_{\alpha, p}  \omega \le c \, u$, and hence 
$ \W_{\alpha, p}  \omega \in L^r(\R^n)$. 
By the estimates in  \cite{HJ}, \cite{JPW} 
again, this is equivalent to the condition
\[
\int_0^\infty \Big[ \frac{\omega(B(x, \rho))}{\rho^{n-\alpha p}}\Big]^{\frac{r}{p-1}} \frac{d \rho}{\rho} \in L^1(\R^n). 
\]
Using the estimate (see \cite{SV2}*{Lemma 4.2})
\begin{equation}\label{kappa-est-B}
[\kappa(B)]^{\frac{q (p-1)}{p-1-q}} \le C \, \int_B u^q d \sigma=C \omega(B),
\end{equation}
where $C=C(\alpha, p, q, n)$, for $B=B(x, \rho)$,  we obtain the following inequality (see \eqref{main-b} in the case $\alpha=1$), 
\[ 
\int_{\R^n} \int_0^\infty \frac{ [\kappa(B(x, \rho))]^{\frac{q r}{p-1-q}}}{\rho^{\frac{(n-\alpha p) r}{p-1}}}
 \frac{d \rho}{\rho} dx<\infty.
\]
This condition obviously implies its dyadic version (see \eqref{main-alpha}), 
\begin{equation}\label{dyad-alpha}
 \sum_{Q \in \mathcal{Q}} \frac{[\kappa(Q)]^{\frac{q r}{p-1-q}}}
{|Q|^{\frac{(n-\alpha p)r-n(p-1)}{n(p-1)}}} <\infty, 
\end{equation}
which also can be deduced independently using the pointwise estimate $\W_{\alpha, p}  \omega \ge 
C \, \W^d_{\alpha, p}  \omega$ (see \cite{HW}) and a version of \eqref{kappa-est-B}  
for cubes $Q$ in place of balls $B$.

Let us next prove that, conversely,  \eqref{dyad-alpha} yields $ \mathbf{K}_{\alpha, p, q}  \sigma \in L^r(\R^n)$. It is enough to show this for the dyadic version, that is,  
$ \mathbf{K}^d_{\alpha, p, q}  \sigma  \in L^r(\R^n)$.

We first  consider the case $0<r\le 1$. Then,  
clearly,  
\[
\Big (\mathbf{K}^d_{\alpha, p, q} \sigma(x)\Big)^r \le \sum_{Q \in \mathcal{Q}} \frac{[\kappa(Q)]^{\frac{q r}{p-1-q}}} {|Q|^{\frac{r(n-\alpha p)}{n(p-1)}}} \chi_Q(x). 
\]
Integrating both sides of the preceding inequality over $\R^n$ with respect to $dx$ 
shows that \eqref{dyad-alpha} yields $ \mathbf{K}^d_{\alpha, p, q}  \sigma \in L^r(\R^n)$. 

We now treat the more difficult case $1<r<\infty$. Let 
\[
\lambda_Q= 
\frac{[\kappa(Q)]^{\frac{q}{p-1-q}}} {|Q|^{\frac{n-\alpha p}{n(p-1)}}}, 
\quad \forall Q \in  \mathcal{Q}. 
\]
Notice that by Theorem \ref{theorem1-3} with $\tau =1$, $q=s$ and $d \sigma=dx$, we have 
that 
$ \mathbf{K}^d_{\alpha, p, q}  \sigma  \in L^r(\R^n)$ if and only if, for some $0<s<r$,
\begin{equation}\label{est-tau}
\sum_{Q \in \mathcal{Q}} \lambda_R \, |R| \, \Big[\frac{1}{|R|} 
\int_R \Big( \sum_{Q \subseteq R} \lambda_Q \chi_Q \Big)^{s} d x\Big ]^{\frac{r-1}{s}}<\infty.
\end{equation}

Let us fix a dyadic cube $R$, and denote by $u_R$ is a solution to the equation 
\[
u = \W_{\alpha, p} (u^q \sigma_R), \quad \text{on} \, \, \R^n, 
\]
where $d \sigma_R = \chi_R d \sigma$ is the restriction 
of $\sigma$ to $R$. Such a solution exists since $\kappa(R)<\infty$; moreover, by \cite{CV2}*{Lemma 4.2 and Corollary 4.3} 
for every $Q\subseteq R$, we have 
\begin{equation}\label{est-d}
\begin{split}
C(\alpha, p, q, n) [\kappa(Q)]^{\frac{q}{p-1-q}}
& \le \Big[\int_Q u_R^q d \sigma\Big]^{\frac{1}{p-1}}\\ & \le \Big[\int_R u_R^q d \sigma\Big]^{\frac{1}{p-1}} \le  [\kappa(R)]^{\frac{q}{p-1-q}}.
\end{split}
\end{equation}

By the first estimate in \eqref{est-d},  we have  
\[
C \, \lambda_Q \le  \Big [ \frac{\int_Q u_R d \sigma}{|Q|^{1-\frac{\alpha p}{n}}}\Big]^\frac{1}{p-1}, \quad Q\subseteq R.
\]
Hence, 
\[
\begin{split}
C \, \int_R \Big( \sum_{Q \subseteq R} \lambda_Q \chi_Q \Big)^{\epsilon} d x
&  \le \int_R \Big[ \sum_{Q \subseteq R} \Big (\frac{\int_Q u_R d \sigma}{|Q|^{1-\frac{\alpha p}{n}}}\Big)^{\frac{1}{p-1}} \chi_Q \Big]^{s} 
d x\\ & \le \int_R  \Big[\W_{\alpha, p} (u^q  d \sigma_R)\Big]^s dx.
\end{split}
 \]
Let $s=\min(1, p-1)<r$. Then by \cite[Lemma 3.1]{SV2}, we can 
estimate the average value of  $\Big[\W_{\alpha, p} (u^q  d \sigma_R)\Big]^s$ 
over $R$: 
\[
\frac{1}{|R|} \int_R  \Big[\W_{\alpha, p} (u_R^q  d \sigma_R)\Big]^s dx \le 
C(\alpha, p, n)  \Big (\frac{\int_R u^q_R d \sigma}{|R|^{1-\frac{\alpha p}{n}}}\Big)^{\frac{s}{p-1}}. 
\]
Further, by the second part of estimate \eqref{est-d}, 
\[
\int_R u^q_R d \sigma \le  [\kappa(R)]^{\frac{q(p-1)}{p-1-q}}.
\]
Combining these estimates, we deduce 
\[
 \Big[\frac{1}{|R|} 
\int_R \Big( \sum_{Q \subseteq R} \lambda_Q \chi_Q \Big)^{s} d x\Big ]^{\frac{r-1}{s}}
\le C 
\frac{[\kappa(R)]^{\frac{q(r-1)}{p-1-q}}}{|R|^{\frac{(n - \alpha p)(r-1)}{n(p-1)}}}.
\]
Consequently, for our choice of $ \lambda_R$, we have 
\begin{equation*}
\begin{split}
& \sum_{Q \in \mathcal{Q}} \lambda_R \, |R| \, \Big[\frac{1}{|R|} 
\int_R \Big( \sum_{Q \subseteq R} \lambda_Q \chi_Q \Big)^{s} d x\Big ]^{\frac{r-1}{s}}\\
& \le  C \sum_{R \in \mathcal{Q}} \frac{[\kappa(R)]^{\frac{q r}{p-1-q}}}
{|R|^{\frac{(n-\alpha p)r-n(p-1)}{n(p-1)}}}
<\infty.
\end{split}
\end{equation*}
Thus, \eqref{est-tau} holds, which  proves that $ \mathbf{K}_{\alpha, p, q}  \sigma \in L^r(\R^n)$. This completes the proofs of Theorems \ref{thm:main-p-laplacian} and \ref{thm:main2}.\qed


\end{document}